\documentclass{amsart}
\usepackage[russian,english]{babel}
\numberwithin{equation}{section}

\newtheorem{thm}{Theorem}[section]
\newtheorem{lem}[thm]{Lemma}
\newtheorem{cor}[thm]{Corollary}

\newtheorem{defin}[thm]{Definition}

\begin{document}
\title{Time-dependent source identification problem for a fractional Telegraph
equation  with the Riemann-Liouville derivative}

\author{Rajapboy Saparbayev}
\address{Institute of Mathematics, Uzbekistan Academy of Science,
Tashkent, Student Town str. 100174} \email{rajapboy1202@gmail.com}

\small

\title[Fractional Telegraph
equation  with the Riemann-Liouville derivative] {Fractional Telegraph
equation  with the Riemann-Liouville derivative }

\begin{abstract}
 The Telegraph equation
$(\partial_{t}^{\rho })^{2}u(x,t)+2\alpha \partial_{t}^{\rho }u(x,t)-u_{xx}(x,t)=f(x,t)$, where $0<t\leq T$ and $0<\rho<1$, with the Riemann-Liouville derivative is considered. Existence and uniqueness theorem for
the solution to the problem under consideration is proved. Inequalities of stability are obtained. The applied method allows us to study a similar problem by taking instead of $d^2/dx^2$ an arbitrary elliptic differential operator $A(x, D)$, having a compact inverse.

\end{abstract}

\maketitle

\section{Introduction}

The fractional integration of order $ \sigma <0 $ of function $ g (t) $ defined on $ [0, \infty) $ has the form (see e.g. \cite{Pskhu}, p.14, \cite{SU}, Chapter 3)
\[
J_t^\sigma g(t)=\frac{1}{\Gamma
(-\sigma)}\int\limits_0^t\frac{g(\xi)}{(t-\xi)^{\sigma+1}} d\xi,
\quad t>0,
\]
provided the right-hand side exists. Here $\Gamma(\sigma)$ is
Euler's gamma function. Using this definition one can define the
 Riemann-Liouville fractional derivative of order $\rho$:

$$
\partial_t^\rho g(t)= \frac{d}{dt}J_t^{\rho-1} g(t).
$$

Note that if $\rho=1$, then the fractional derivative coincides with
the ordinary classical derivative of the first order: $\partial_t g(t)= (d/dt) g(t)$.

Let $\rho\in(0,1) $ be a fixed number and $\Omega =(0,\pi) \times (0, T]$.
Consider the following initial-boundary value problem for the Telegraph equation
\begin{equation}\label{prob1}
\left\{
\begin{aligned}
&(\partial_t^\rho)^{2} u(x,t)+2\alpha \partial_{t}^{\rho}u(x,t)- u_{xx}(x,t) =f(x,t),\quad (x,t)\in \Omega;\\
&u(0,t)=u(\pi, t)=0, \quad 0\leq t\leq T;\\
&\lim\limits_{t\rightarrow 0}J_t^{\rho-1}(\partial^{\rho}_{t} u(x,t))=\varphi_{0}(x), \quad 0\leq x\leq \pi;\\
&\lim\limits_{t\rightarrow 0}J_t^{\rho-1}u(x,t)=\varphi_{1}(x), \quad 0\leq x\leq \pi,
\end{aligned}
\right.
\end{equation}
where $t^{1-\rho}f(x,t)$ and $\varphi_{0}(x), \varphi_{1}(x)$ are continuous functions in the closed domain $\overline{\Omega}$. This problem is also called the \textit{forward problem}.

Note that one can study the above equation with the operator $\partial_t^{2\rho}$ instead of $(\partial_t^{\rho})^2$. However, these two operators are not the same and the corresponding problems are completely different. As a simple example, we can take the function: 
\begin{equation*}
    u(t)=t^{\rho -1}, \qquad t>0.
\end{equation*}
It is easy to see that 
\begin{equation*}
    (\partial^{\rho}_{t})^{2}=\partial^\rho(\partial^\rho u)\neq \partial^{2\rho}u.
\end{equation*}
\begin{defin}\label{forward}If function $u(t)$ with the properties $(\partial_{t}^{\rho })^{2}u(x,t), \partial_{t}^{\rho}u(x,t), u_{xx}(x,t)\in C(\Omega)$ and  $t^{1-\rho}\partial^{\rho}_{t}u(x,t)$, $t^{1-\rho}u(x,t)\in C(\overline{\Omega})$  satisfies conditions (\ref{prob1}) then it is called \textbf{the  solution} of the forward problem.
\end{defin}

We note the following property of the Riemann-Liouville integrals, which simplifies the verification of the initial condition in problem (\ref{prob1}) (see e.g. \cite{Pskhu}, p.104):

\begin{equation} \label{property}
	\lim\limits_{t\rightarrow +0}J_t^{\alpha-1} g(t) = \Gamma
	(\alpha)\lim\limits_{t\rightarrow +0} t^{1-\alpha} g(t).
\end{equation}
From here, in particular, it follows that the solution of the forward problem can have a
singularity at zero $t = 0$ of order $t^{\rho-1}$.

Taking into account the boundary conditions in problem (\ref{prob1}), it is convenient for us to introduce the H\"older classes as follows. Let $\omega_g(\delta)$ be the modulus of continuity of function $g(x)\in C[0, \pi]$, i.e.
\[
\omega_g(\delta)=\sup\limits_{|x_1-x_2|\leq\delta}
|g(x_1)-g(x_2)|, \quad x_1, x_2\in [0,\pi].
\]
If $\omega_g(\delta)\leq C \delta^a$ is true for some $a>0$, where $C$ does not depend on $\delta$ and $g(0)=g(\pi)=0$, then $g(x)$ is said to belong to the H\"older class $C^a[0,\pi]$. Let us denote the smallest of all such constants $C$ by $||g||_{C^a[0,\pi]}$. Similarly, if the continuous function $h(x, t)$ is defined on $[0,\pi]\times[0, T]$, then the value
\[
\omega_h(\delta ;t)=\sup\limits_{|x_1-x_2|\leq\delta} |h(x_1,
t)-h(x_2, t)|, \quad x_1, x_2\in [0,\pi]
\]
is the modulus of continuity of function $h(x, t)$ with respect to the variable
$x$. In case when $\omega_h(\delta; t)\leq C \delta^a$, where
$C$ does not depend on $t$ and $\delta$ and $h(0,t)=h(\pi,t)=0,\,\, t\in [0,T]$, then we say that
$h(x, t)$ belongs to the H\"older class $C_x^a(\overline{\Omega})$. Similarly, we denote the smallest constant $C$ by  $||h||_{C_x^a(\overline{\Omega})}$.

Let $C_{2,x}^a(\overline{\Omega})$ denote the class of functions $h(x,t)$ such that $h_{xx}(x,t)\in C_x^a(\overline{\Omega})$  and  $h(0,t)=h(\pi,t)=0,\, t\in [0,T]$. Note that condition $h_{xx}(x,t)\in C_x^a(\overline{\Omega})$  implies that $h_{xx}(0,t)=h_{xx}(\pi,t)=0,\,\, t\in [0,T]$. For a function of one variable $g(x)$, we introduce classes $C_{2}^a[0, \pi]$ in a similar way.

\begin{thm}\label{main} Let $\alpha>0$, $a>\frac{1}{2}$ and the following conditions be satisfied
	
	\begin{enumerate}
		\item
		$t^{1-\rho}f(x,t)\in C^a_x(\overline{\Omega})$,
		\item$\varphi_{0}(x),\varphi_{1}(x) \in C^a_{2}[0, \pi]$.
		
	\end{enumerate}
	Then the forward problem has a unique solution.

Moreover, there is a constant $C>0$ such that the following stability estimate
\[
   ||t^{1-\rho}(\partial_{t}^{\rho})^{2} u(x,t)||_{C(\overline{\Omega})} +||t^{1-\rho}\partial_t^\rho u(x,t)||_{C(\overline{\Omega})} 
   + ||t^{1-\rho}u_{xx}(x,t)||_{C(\overline{\Omega})}
  \]
  \[ 
   \leq C \bigg[||\varphi_{0}(x)||_{C^{a}_{2}[0,\pi]}+||\varphi_{1}||_{C^{a}_{2}[0,\pi]}
+ ||t^{1-\rho}f(x,t)||_{C^a_x(\overline{\Omega})}\bigg],
\]
holds.
\end{thm}

The theory of differential equations involving fractional derivatives, both under equation and boundary conditions, has developed rapidly over the past few decades in both mathematics and applied sciences (see, for example, \cite{SU}-\cite{AOLob}).The problem of identification of fractional order of the model was considered by many researchers. Note that all the publications assumed the fractional derivative of order $0 < \rho < 1$ in the sense of Caputo and studied mainly the uniqueness problem. In this paper, the fractional-order Telegraph equation in the Riemann-Liouville sense with respect to time is studied. Several studies were done to work aut the telegraph equation numerically or analytically as in  (see \cite{Hashmi}-\cite{Huang}).Thus in the works \cite{Orsingher} ($\rho=1/2 $ if), \cite{Beghin} (in case of fraction
derivatives of the rational order $ \rho = m / n$ with $ M < n$), fundamental solutions for the Telegraph equation of the Caputo operator have been studied. Such as, it shows in the theory of superconducting electrodynamics, where it illustrates the propagation
electromagnetic waves in superconducting media (see, e.g.\cite{Wallace}).
In \cite{Jordan}, the propagation of digital and analog signals through media which, in general, are
both dissipative and dispersive is modeled using the telegraph equation. 

Some applications of the telegraph equation to the theory of random walks are contained in \cite{Banasiak}. Another field of application of the telegraph equation is the biological sciences (see, e.g. \cite{Debnath}, \cite{Goldstein}).

The closest papers to this article are \cite{RC},\cite{AS} but they studied the Caputo.
It is important to note that the method proposed here, based on the Fourier method, is applicable to the equation in (\ref{prob1}) with an arbitrary elliptic differential operator $A(x, D)$ instead of $d^2/dx^2$, if only the corresponding spectral problem has a complete system of orthonormal eigenfunctions in $ L_2(G), \, G\subset R^N$.

\section{Preliminaries}

In this section, we  recall some information about the Mittag-Leffler functions, differential and integral equations, which we will use in the following sections.

For $0 < \rho < 1$ and an arbitrary complex number $\mu$, by $
E_{\rho, \mu}(z)$ we denote the Mittag-Leffler function  of complex argument $z$ with two
parameters:
\begin{equation}\label{ml}
E_{\rho, \mu}(z)= \sum\limits_{k=0}^\infty \frac{z^k}{\Gamma(\rho
k+\mu)}.
\end{equation}
If the parameter $\mu =1$, then we have the classical
Mittag-Leffler function: $ E_{\rho}(z)= E_{\rho, 1}(z)$.
Prabhakar (see, \cite{Prah}) introduced  the function $E^{\gamma}_{\rho, \mu}(z)$ of the form
\begin{equation}\label{ml1}
E^{\gamma}_{\rho, \mu}(z)= \sum\limits_{k=0}^\infty \frac{(\gamma)_{k}}{\Gamma(\rho k+\mu)}\cdot\frac{z^{k}}{k!},
\end{equation}
where $z \in C$, $\rho$, $\mu$ and $\gamma$ are arbitrary positive constants, and $(\gamma)_{k}$ is the Pochhammer
symbol. When $\gamma = 1$, one has $E^{1}_{\rho,\mu}(z)=E_{\rho,\mu}(z)$. We also have \cite{Prah}
\begin{equation}\label{Prahma}
E^{2}_{\rho,\mu}(z)=\frac{1}{\rho}\big[E_{\rho,\mu-1}(z)+(1-\rho+\mu)E_{\rho,\mu}(z)\big].
\end{equation}
Since $E_{\rho, \mu}(z)$  is an analytic function of $z$, then it is bounded for $|z|\leq 1$. On the other hand the well known asymptotic estimate of the
Mittag-Leffler function has the form (see, e.g.,
\cite{Dzh66}, p. 133):
\begin{lem}\label{ml} Let  $\mu$  be an arbitrary complex number. Further let $\beta$ be a fixed number, such that $\frac{\pi}{2}\rho<\beta<\pi \rho$, and $\beta \leq |\arg z|\leq \pi$. Then the following asymptotic estimate holds
\[
E_{\rho, \mu}(z)= -\frac{z^{-1}}{\Gamma(\rho-\mu)} + O(|z|^{-2}),
\,\, |z|>1.
\]
\end{lem}
\begin{cor}\label{cor} Under the conditions of Lemma \ref{ml} one has
\[
|E_{\rho,\mu}(z)|\le \frac{M}{1+|z|}, \quad |z|\ge0,
\]
where  $M$-constant, independent of $z$.
\end{cor}
We also use the following estimate for sufficiently large $\lambda>0$  and $\alpha>0$, $0<\epsilon<1$:
\begin{equation}\label{Large}
  |t^{\rho-1}E_{\rho,\mu}(-(\alpha-\sqrt{\alpha^{2}-\lambda})t^{\rho})|\leq\frac{t^{\rho-1}M}{1+\sqrt{\lambda}t^{\rho}}\leq M\lambda^{\epsilon-\frac{1}{2}}t^{2\epsilon\rho-1}, \quad  t>0,
\end{equation}
which is easy to verify. Indeed, let $(\lambda)^{\frac{1}{2}} t^{\rho}<1$, then $t<\lambda^{-\frac{1}{2\rho}}$ and
\[
t^{\rho-1}=t^{\rho-2\epsilon\rho}t^{2\epsilon\rho-1}<\lambda^{\epsilon-\frac{1}{2}}t^{2\epsilon\rho-1}.
\]
If $(\lambda)^{\frac{1}{2}} t^{\rho}\geq 1$, then $\lambda^{-\frac{1}{2}}\leq t^{\rho}$ and
\[
\lambda^{-\frac{1}{2}}t^{-1}=\lambda^{\epsilon-\frac{1}{2}}\lambda^{-\epsilon}t^{-1}\leq \lambda^{\epsilon-\frac{1}{2}}t^{2\rho\epsilon-1}.
\]
\begin{lem}\label{Dervitavi1}If $\rho>0$ and $\lambda \in \mathbb{C} $, then
\[
 \partial^{\rho}_{t}\left(t^{\rho-1}E_{\rho,\rho}(\lambda t^{\rho})\right)=t^{\rho-1}\lambda E_{\rho,\rho}(\lambda t^{\rho}),   \quad t>0 ,
\]
\end{lem}
The proof of this lemme for $\lambda\in\mathbb{R}$ can be found in \cite{AFFF}. In   the complex case the similar ideas will lead us to the same conclusion.

\begin{lem}\label{Dervitavit2}If $\rho>0$ and $\lambda \in \mathbb{C} $, then (see \cite{Pang})
\[
\partial^{\rho}_{t}\left(t^{2\rho-1}E^{2}_{\rho,2\rho}(\lambda t^{\rho}) \right)=t^{\rho-1}E^{2}_{\rho,\rho}(\lambda t^{\rho}),   \quad t>0 .
\]
\end{lem}
\begin{lem}\label{Dervitavi3}If $\rho>0$ and $\lambda \in \mathbb{C} $ and $g(t)\in C[0,T]$  then (see \cite{Kilbas})
\[
 \partial^{\rho}_{t}\left(\int_{0}^{t}(t-\tau)^{2\rho-1}E^{2}_{\rho,2\rho}(\lambda(t-\tau)^{\rho})g(\tau)d\tau\right)=\int_{0}^{t}(t-\tau)^{\rho-1}E_{\rho,\rho}^{2}(\lambda(t-\tau)^{\rho})g(\tau)d\tau.
\]
\end{lem}
\begin{lem}\label{Dervitavi4} The solution to the Cauchy problem
\begin{equation}\nonumber
\begin{cases}
  & \partial_{t}^{\rho }u(t)-\lambda u(t)=f(t),\quad  0<t\le T; \\
 &\lim\limits_{t\to 0} J^{\rho-1}_{t}u(t)=0, \\
\end{cases}
\end{equation}
with $0<\rho<1$ and $\lambda\in\mathbb{C} $ has the form
\[
u(t)=\int_{0}^{t}(t-\tau)^{\rho-1}E_{\rho,\rho}(\lambda(t-\tau)^{\rho})f(\tau)d\tau.
\]
\end{lem}
The proof of this lemme for $\lambda\in\mathbb{R}$ can be found in \cite{Kil},p.224. In   the complex case the similar ideas will lead us to the same conclusion.

\begin{lem}\label{du}Let $t^{1-\rho}g(t)\in C[0,T]$ and $\varphi_{0},\varphi_{1} $ be known numbers. Then the unique solution of the Cauchy problem
\begin{equation}\label{prob12}
\begin{cases}
  & (\partial_{t}^{\rho })^{2}y(t)+2\alpha \partial_{t}^{\rho }y(t)+\lambda y(t)=g(t), \quad  0<t\le T; \\
 & \underset{t\to 0}{\mathop{\lim }}\,J^{\rho-1}_{t}\partial_{t}^{\rho }y(t)={{\varphi }_{0}},\\
 & \underset{t\to 0}{\mathop{\lim }}\,J^{\rho-1}_{t}y(t)={{\varphi }_{1}}, \\
\end{cases}
\end{equation}
has the form
\begin{equation}\label{1}y(t)=
\begin{cases}
  & y_{1}(t),\quad \alpha^{2}\neq \lambda; \\
 & y_{2}(t),\quad \alpha^{2}=\lambda.
\end{cases}
\end{equation}
Here
$$
y_{1}(t)=\frac{\sqrt{\alpha^{2}-\lambda}-\alpha}{2\sqrt{\alpha^{2}-\lambda}}t^{\rho-1}E_{\rho,\rho}((-\alpha-\sqrt{\alpha^{2}-\lambda})t^{\rho})\varphi_{1}+\frac{\sqrt{\alpha^{2}-\lambda}+\alpha}{2\sqrt{\alpha^{2}-\lambda}}t^{\rho-1}E_{\rho,\rho}((-\alpha+\sqrt{\alpha^{2}-\lambda})t^{\rho})\varphi_{1}
$$
$$
+\frac{1}{2\sqrt{\alpha^{2}-\lambda}}\left(t^{\rho-1}E_{\rho,\rho}((-\alpha+\sqrt{\alpha^{2}-\lambda})t^{\rho})-t^{\rho-1}E_{\rho,\rho}((-\alpha-\sqrt{\alpha^{2}-\lambda})t^{\rho})  \right)\varphi_{0}
$$
$$
+\frac{1}{2\sqrt{\alpha^{2}-\lambda}}\int_{0}^{t}(t-\tau)^{\rho-1}E_{\rho,\rho}((-\alpha+\sqrt{\alpha^{2}-\lambda})(t-\tau)^{\rho})g(\tau)d\tau
$$
$$
-\frac{1}{2\sqrt{\alpha^{2}-\lambda}}\int_{0}^{t}(t-\tau)^{\rho-1}E_{\rho,\rho}((-\alpha-\sqrt{\alpha^{2}-\lambda})(t-\tau)^{\rho})g(\tau)d\tau.
$$
\[
y_{2}(t)=t^{\rho-1}E_{\rho,\rho}(-\alpha t^{\rho})\varphi_{1}+\alpha t^{2\rho-1} E_{\rho,2\rho}^{2}(-\alpha t^{\rho})\varphi_{1}+t^{2\rho-1}E_{\rho,2\rho}^{2}(-\alpha t^{\rho})\varphi_{0}
\]
\[
+\int_{0}^{t}(t-\tau)^{2\rho-1}E^{2}_{\rho,2\rho}(-\alpha(t-\tau)^{\rho})g(\tau)d\tau.
\]
\end{lem}
\begin{proof} We use the Laplace transform to prove the lemma. Let us remind that the Laplace transform of a function $f(t)$ is defined as (see \cite{Mainardi})
\[
L[f](p)=\hat{f}(p)=\int_{0}^{\infty}e^{-pt}f(t)dt.
\]
The inverse Laplace transform is defined by
\[
L^{-1}[\hat{f}](t)=\frac{1}{2\pi i}\int_{C}e^{pt}\hat{f}(p)dp,
\]
where $C$ is a contour parallel to the imaginary axis and to the right of the singularities of  $\hat{f}$.

Let us apply the Laplace transform to equation (\ref{prob12}).  Then equation (\ref{prob12}) becomes:
\[
p^{2\rho}\hat{y}(p)+2\alpha p^{\rho}\hat{y}(p)+\lambda\hat{y}(p)-p^{\rho}\lim_{t \to 0}J^{\rho-1}_{t}y(t)-\lim_{t\to 0}J^{\rho-1}_{t}\partial _{t}^{\rho}y(t)-2\alpha \lim_{t \to 0}J^{\rho-1}_{t}y(t)=\hat{g}(p),
\]
it follows from this
\begin{equation}\label{laplas}
\hat{y}(p)=\frac{\hat{g}(p)+p^{\rho}\lim\limits_{t\to 0}J^{\rho-1}_{t}y(t)+\lim\limits_{t\to 0}J^{\rho-1}_{t}\partial _{t}^{\rho}y(t)+2\alpha\lim\limits_{t\to 0}J^{\rho-1}_{t}y(t)}{p^{2\rho}+2\alpha p^{\rho}+\lambda}.
\end{equation}
\textbf{Case 1}. Let  $\alpha^{2}\neq \lambda$.

Write $\hat{y}(p)=\hat{y}_{0}(p)+\hat{y}_{1}(p)$, where
\[
\hat{y}_{0}(p)=\frac{p^{\rho}\varphi_{1}+\varphi_{0}+2\alpha \varphi_{1}}{{p^{2\rho}+2\alpha p^{\rho}}+\lambda}, \quad \hat{y}_{1}(p)=\frac{\hat{g}(p)}{{p^{2\rho}+2\alpha p^{\rho}}+\lambda},
\]
furthemore
\[
y(t)=L^{-1}[\hat{y}_{0}(p)]+L^{-1}[\hat{y}_{1}(p)].
\]
For the first term of $y(t)$, the inverse can be obtained by splitting the function $\hat{y}_{0}$ into simpler functions:
\begin{equation}\label{y0}
L^{-1}[\hat{y}_{0}(p)]=L^{-1} \left[ \frac{p^{\rho}}{p^{2\rho}+2\alpha p^{\rho}+\lambda}\right]\varphi_{1}+L^{-1}\left[\frac{1}{p^{2\rho}+2\alpha p^{\rho}+\lambda}\right](\varphi_{0}+2\alpha \varphi_{1}).
\end{equation}
The following simple observations show that
\begin{equation*}
\begin{split}
    L^{-1}\bigg[&\frac{1}{p^{2\rho}+2\alpha p^{\rho}+\lambda}\bigg] = L^{-1}\bigg[\frac{1}{2\sqrt{\alpha^{2}-\lambda}}\bigg(\frac{1}{p^{\rho}+\alpha-\sqrt{\alpha^{2}-\lambda}}-\frac{1}{p^{\rho}+\alpha+\sqrt{\alpha^{2}-\lambda}}\bigg)\bigg]\\
    &=\frac{1}{2\sqrt{\alpha^{2}-\lambda}}L^{-1}\bigg[\frac{1}{p^{\rho}+\alpha-\sqrt{\alpha^{2}-\lambda}}\bigg]-\frac{1}{2\sqrt{\alpha^{2}-\lambda}}L^{-1}\bigg[\frac{1}{p^{\rho}+\alpha+\sqrt{\alpha^{2}-\lambda}}\bigg],
\end{split}
\end{equation*}
or (see \cite{RC})
\[
 L^{-1}\bigg[\frac{1}{p^{2\rho}+2\alpha p^{\rho}+\lambda}\bigg]=\frac{t^{\rho-1}}{2\sqrt{\alpha^{2}-\lambda}}E_{\rho,\rho}\bigg((-\alpha+\sqrt{\alpha^{2}-\lambda})t^{\rho}\bigg)-\frac{t^{\rho-1}}{2\sqrt{\alpha^{2}-\lambda}}E_{\rho,\rho}\bigg((-\alpha-\sqrt{\alpha^{2}-\lambda})t^{\rho}\bigg).
\]
We also find the inverse Laplace transform of the first term \eqref{y0} by following computations:
\[
 L^{-1}\bigg[\frac{p^{\rho}}{p^{2\rho}+2\alpha p^{\rho}+\lambda}\bigg]=\frac{1}{2}L^{-1}\left[\frac{1}{p^{\rho}+\alpha+\sqrt{\alpha^{2}-\lambda}}\right]+\frac{1}{2}L^{-1}\left[\frac{1}{p^{\rho}+\alpha-\sqrt{\alpha^{2}-\lambda}}\right]
\]
\[
+\frac{\alpha}{2\sqrt{\alpha^{2}-\lambda}}L^{-1}\left[\frac{1}{p^{\rho}+\alpha+\sqrt{\alpha^{2}-\lambda}}\right]-\frac{\alpha}{2\sqrt{\alpha^{2}-\lambda}}L^{-1}\left[\frac{1}{p^{\rho}+\alpha-\sqrt{\alpha^{2}-\lambda}}\right]
\]
\[
=\frac{t^{\rho-1}}{2}E_{\rho,\rho}\left(\left(-\alpha-\sqrt{\alpha^{2}-\lambda}\right)t^{\rho}\right)+\frac{t^{\rho-1}}{2}E_{\rho,\rho}\left(\left(-\alpha+\sqrt{\alpha^{2}-\lambda}\right)t^{\rho}\right)
\]
\[
+\frac{\alpha t^{\rho-1}}{2\sqrt{\alpha^{2}-\lambda}}\left (E_{\rho,\rho}\left(\left(-\alpha-\sqrt{\alpha^{2}-\lambda}\right)t^{\rho}\right)-E_{\rho,\rho}\left(\left(-\alpha+\sqrt{\alpha^{2}-\lambda}\right)t^{\rho}\right)\right)
\]
For the second term of $y(t)$ one can obtain the inverse by splitting the function $\hat{y}_{1}$ into the simpler functions:
\begin{equation}\label{y1}
L^{-1}[\hat{y}_{1}(p)]=L^{-1}\bigg[\frac{\hat{g}(p)}{p^{2\rho}+2\alpha p^{\rho}+\lambda}\bigg]=L^{-1}\bigg[\frac{1}{p^{2\rho}+2\alpha p^{\rho}+\lambda}\bigg]\ast L^{-1}[\hat{g}(p)].
\end{equation}
By $f \ast g$ we denoted the Laplace convolution of functions defined by  $(f\ast g)(t)=\int_{0}^{t}f(\tau)g(t-\tau)d\tau$.

Plugging this function into \eqref{y1} and combining it with \eqref{y0} we have
$$
y(t)=\frac{\sqrt{\alpha^{2}-\lambda}-\alpha}{2\sqrt{\alpha^{2}-\lambda}}t^{\rho-1}E_{\rho,\rho}((-\alpha-\sqrt{\alpha^{2}-\lambda})t^{\rho})\varphi_{1}+\frac{\sqrt{\alpha^{2}-\lambda}+\alpha}{2\sqrt{\alpha^{2}-\lambda}}t^{\rho-1}E_{\rho,\rho}((-\alpha+\sqrt{\alpha^{2}-\lambda})t^{\rho})\varphi_{1}
$$
$$
+\frac{1}{2\sqrt{\alpha^{2}-\lambda}}\left(t^{\rho-1}E_{\rho,\rho}((-\alpha+\sqrt{\alpha^{2}-\lambda})t^{\rho})-t^{\rho-1}E_{\rho,\rho}((-\alpha-\sqrt{\alpha^{2}-\lambda})t^{\rho})  \right)\varphi_{0}
$$
$$
+\frac{1}{2\sqrt{\alpha^{2}-\lambda}}\int_{0}^{t}(t-\tau)^{\rho-1}E_{\rho,\rho}((-\alpha+\sqrt{\alpha^{2}-\lambda})(t-\tau)^{\rho})g(\tau)d\tau
$$
$$
-\frac{1}{2\sqrt{\alpha^{2}-\lambda}}\int_{0}^{t}(t-\tau)^{\rho-1}E_{\rho,\rho}((-\alpha-\sqrt{\alpha^{2}-\lambda})(t-\tau)^{\rho})g(\tau)d\tau.
$$

\textbf{Case 2}. Let  $\alpha^{2}=\lambda$. In this case (\ref{laplas}) has the following form
\[
\hat{y}(p)=\frac{\hat{g}(p)+p^{\rho}\lim\limits_{t\to 0}J^{\rho-1}_{t}y(t)+\lim\limits_{t\to 0}J^{\rho-1}_{t}\partial_{t}^{\rho}y(t)+2\alpha\lim\limits_{t\to 0}J^{\rho-1}_{t}y(t)}{(p^{\rho}+\alpha)^{2}}.
\]

Passing to the inverse Laplace transform (see \cite{Mainardi},p.226,E67):

\[
y(t)=L^{-1}\left[\frac{1}{p^{\rho}+\alpha}\right]\lim_{t\to 0}J^{\rho-1}_{t}y(t)-L^{-1}\left[\frac{\alpha}{(p^{\rho}+\alpha)^{2}}\right]\lim_{t\to 0}J^{\rho-1}_{t}y(t)
\]
\[
+L^{-1}\left[\frac{1}{(p^{\rho}+\alpha)^{2}}\right]\lim_{t\to 0}J^{\rho-1}_{t}\partial_{t}^{\rho}y(t)+L^{-1}\left[\frac{2\alpha}{(p^{\rho}+\alpha)^{2}}\right]\lim_{t\to 0}J^{\rho-1}_{t}y(t)+L^{-1}\left[\frac{1}{(p^{\rho}+\alpha)^{2}}\right]\ast L^{-1}[\hat{g}(p)]
\]
one has
\[
y(t)=t^{\rho-1}E_{\rho,\rho}(-\alpha t^{\rho})\varphi_{1}+\alpha t^{2\rho-1} E_{\rho,2\rho}^{2}(-\alpha t^{\rho})\varphi_{1}+t^{2\rho-1}E_{\rho,2\rho}^{2}(-\alpha t^{\rho})\varphi_{0}
\]
\[
+\int_{0}^{t}(t-\tau)^{2\rho-1}E^{2}_{\rho,2\rho}(-\alpha(t-\tau)^{\rho})g(\tau)d\tau.
\]
\end{proof}

Let us denote by $A$ the operator $-d^2/dx^2$ with the domain $D(A) = \{v(x)\in W_2^2(0, \pi): v(0)=v(\pi)=0\}$, where $W_2^2(0, \pi)$ - the standard Sobolev space. Operator $A$ is selfadjoint in $L_2(0, \pi)$ and has the complete in $L_2(0, \pi)$ set of  eigenfunctions $\{v_k(x) = \sin kx\}$ and eigenvalues $\lambda_k=k^2$, $k=1,2,...$.

Consider the operator $E_{\rho, \mu} (t A)$, defined by the spectral theorem of J. von Neumann:
\[
E_{\rho, \mu} (t A)h(x,t) = \sum\limits_{k=1}^\infty E_{\rho,\mu} (t \lambda_k) h_k(t) v_k(x),
\]
here and everywhere below, by $h_k(t)$  we will denote the Fourier coefficients of a function $h(x,t)$: $h_k(t)=(h(x,t),v_k)$, $(\cdot, \cdot)$ stands for the scalar product in $L_2(0, \pi)$. This series converges in the $L_2(0, \pi)$ norm. But we need to investigate the uniform convergence of this series in $\Omega$. To do this, we recall the following statement.

\begin{lem}\label{Zyg}Let  $g\in C^a[0, \pi]$.
Then for any $\sigma\in [0, a-1/2)$ one has
\[
\sum\limits_{k=1}^\infty k^\sigma|g_k|<\infty.
\]	
\end{lem}

For $\sigma=0$ this assertion coincides with the well-known theorem of S. N. Bernshtein on the absolute convergence of trigonometric series and is proved in exactly the same way as this theorem. For the convenience of readers, we recall the main points of the proof (see, e.g. \cite{Zyg}, p. 384).

\begin{proof}

In theorem (3.1) of A. Zygmund \cite{Zyg}, p. 384, it is proved
that for an arbitrary function $g(x)\in C[0,\pi]$, with the properties
$g(0)=g(\pi)=0$, one has the estimate
\[
\sum\limits_{k=2^{n-1}+1}^{2^n} |g_k|^2 \leq
\omega_g^2\bigg(\frac{1}{2^{n+1}}\bigg).
\]
Therefore, if $\sigma\geq 0$, then by the Cauchy-Bunyakovsky inequality
\[
\sum\limits_{k=2^{n-1}+1}^{2^n} k^\sigma |g_k| \leq
\bigg(\sum\limits_{k=2^{n-1}+1}^{2^n}
|g_k|^2\bigg)^{\frac{1}{2}}\bigg(\sum\limits_{k=2^{n-1}+1}^{2^n}
k^{2\sigma}\bigg)^{\frac{1}{2}}\leq C
2^{n(\frac{1}{2}+\sigma)}\omega_g\bigg(\frac{1}{2^{n+1}}\bigg),
\]
and finally
\[
\sum\limits_{k=2}^{\infty} k^\sigma|g_k|=
\sum\limits_{n=1}^\infty\sum\limits_{k=2^{n-1}+1}^{2^n} k^\sigma|g_k| \leq  C \sum\limits_{n=1}^\infty
2^{n(\frac{1}{2}+\sigma)}\omega_g\bigg(\frac{1}{2^{n+1}}\bigg).
\]
Obviously, if $\omega_g(\delta)\leq C \delta^a$, $a>1/2$ and $0<\sigma<a - 1/2$, then
the last series converges:
\[
\sum\limits_{k=2}^{\infty} k^\sigma|g_k|\leq C ||g||_{C^a[0,\pi]}.
\]

\end{proof}

\begin{lem}\label{E} Let $\alpha>0$. Then for any $g(x,t)\in C_x^a(\overline{\Omega})$  one has $E_{\rho,\mu}(-S t^{\rho})g(x,t)\in C_x^a(\overline{\Omega})$
and  $SE_{\rho,\mu}(-St^{\rho})g(x,t)\in C_x^a({\Omega})$.  Moreover,the following estimates hold:
    \begin{equation}\label{ES}
    ||E_{\rho, \mu} (-St^{\rho})g(x,t)||_{C(\overline{\Omega})} \leq M ||g||_{C_x^a(\overline{\Omega})};
    \end{equation}
\begin{equation}\label{SES1}
   \bigg |\bigg|SE_{\rho, \mu} (-St^{\rho})g(x,t)\bigg|\bigg|_{C[0, \pi]} \leq M t^{-\rho} ||g||_{C_x^a(\overline{\Omega})}, \,\, t>0.
    \end{equation}
If
$g(x,t)\in C_{2,x}^a(\overline{\Omega})$,
then

\begin{equation}\label{SES2}
   \bigg |\bigg|SE_{\rho, \mu} (-St^{\rho})g(x,t)\bigg|\bigg|_{C(\overline{\Omega})} \leq C_{1} ||g_{xx}||_{C_x^a(\overline{\Omega})};
\end{equation}

\begin{equation}\label{AES}
   \bigg |\bigg|\frac{\partial^2}{\partial x^2} E_{\rho, \mu} (-St^{\rho})g(x,t)\bigg|\bigg|_{C(\overline{\Omega})}\leq C_{2}||g_{xx}||_{C_x^a(\overline{\Omega})}.
 \end{equation}

Here $S$ has two states: $S^{-}$ and $S^{+}$,
\[
S^{-}=\alpha I-(\alpha^{2}I-A)^{\frac{1}{2}},\quad  S^{+}=\alpha I+(\alpha^{2}I-A)^{\frac{1}{2}}.
\]

\end{lem}
\begin{proof}By definition one has
    \[
    ||E_{\rho, \mu} (-S^{-}t^{\rho})g(x,t)||^{2}=\bigg|\bigg|\sum\limits_{k=1}^\infty E_{\rho, \mu} \left(-t^{\rho}\left(\alpha-\sqrt{\alpha^{2}-\lambda_{k}} \right)\right) g_k(t) v_k(x)\bigg|\bigg|^{2}.
    \]
    Corollary \ref{ml1} and Lemma \ref{Zyg} imply
    \[
    ||E_{\rho, \mu} (-S^{-}t^{\rho})g(x,t)||^{2}\leq M^{2} \sum\limits_{k=1}^\infty \bigg|\frac{ g_k(t)}{1+t^{\rho}\left|\alpha-\sqrt{\alpha^{2}-\lambda_{k}} \right|}\bigg|^{2}\leq M^{2} ||g||^{2}_{C_x^a(\overline{\Omega})}.
    \]
On the other hand,
    \[
   ||S^{-}E_{\rho, \mu} (-S^{-}t^{\rho})g(x,t)||^{2}\leq M^{2} \sum\limits_{k=1}^\infty \bigg|\frac{\left|\alpha-\sqrt{\alpha^{2}-\lambda_{k}}\right| g_k(t)}{1+t^{\rho} \left|\alpha-\sqrt{\alpha^{2}-\lambda_{k}}\right|}\bigg|^{2}\leq M^{2} t^{-2\rho} ||g||^{2}_{C_x^a(\overline{\Omega})}, \,\, t>0.
    \]

    If $g(x,t)\in C_{2,x}^a(\overline{\Omega})$, then $g_k(t)=-\lambda_k^{-1} (g_{xx})_k(t)$. Therefore,
 \[
   ||S^{-}E_{\rho, \mu} (-S^{-}t^{\rho})g(x,t)||^{2}\leq M^{2} \sum\limits_{k=1}^\infty \bigg|\frac{\left|\alpha-\sqrt{\alpha^{2}-\lambda_{k}}\right| g_k(t)}{1+t^{\rho} \left|\alpha-\sqrt{\alpha^{2}-\lambda_{k}}\right|}\bigg|^{2},
   \]

\[
u_{\lambda_{k}}(t)=\frac{ |\alpha-\sqrt{\alpha^{2}-\lambda_{k}}|^{2}|g_k(t)|^{2}}{(1+t^{\rho}|\alpha-\sqrt{\alpha^{2}-\lambda_{k}|})^{2}} \quad  \underset{\lambda_{k}\to \infty}{\sim} \quad v_{\lambda_{k}}(t)=\frac{ \lambda_{k}|g_k(t)|^{2}}{(1+t^{\rho}\sqrt{\lambda_{k}})^{2}},
\]

\[
\sum_{k=1}^{\infty}v_{\lambda_{k}}(t)=\sum_{k=1}^{\infty}\frac{ \lambda_{k}|g_k(t)|^{2}}{(1+t^{\rho}\sqrt{\lambda_{k}})^{2}}\leq\sum_{k=1}^{\infty}\lambda^{2}_{k}|g_k(t)|^{2}\frac{1}{\lambda_{k}}\leq\sum_{k=1}^{\infty}\lambda^{2}_{k}|g_k(t)|^{2}\sum_{k=1}^{\infty}\frac{1}{k^{2}}\leq C||g_{xx}||^{2}_{C_x^a(\overline{\Omega})}.
\]

 (We have used the notation here:  ${u_{\lambda_{k}}} \underset{\lambda_{k}\to \infty}{\sim} {v_{\lambda_{k}}}$ means
$
 \lim\limits_{\lambda_{k}\to \infty}\frac{{u_{\lambda_{k}}}}{{v_{\lambda_{k}}}}=1 $).

Therefore
\[
||S^{-}E_{\rho, \mu} (-t^{\rho} S^{-})g(x,t)||_{C(\overline{\Omega})}\leq C_{1} ||g_{xx}||_{C_x^a(\overline{\Omega})}.
\]
 Obviously
 \[
 \bigg |\bigg|\frac{\partial^2}{\partial x^2} E_{\rho, \mu} (-S^{-}t^{\rho})g(x,t)\bigg|\bigg|_{C(\overline{\Omega})}\leq C_{2}||g_{xx}||_{C_x^a(\overline{\Omega})}.
   \]
  A similar estimate is proved in exactly the same way with the operator $S^{-}$ replaced by the operator $S^{+}$.

\end{proof}
\begin{lem} \label{ER} Let  $\alpha>0$ and   $\lambda_{k}\neq \alpha^{2}$,  for all $k$. Then for any $g(x,t)\in C_x^a(\overline{\Omega})$
 one has $R^{-1}E_{\rho,\mu}(-St^{\rho})g(x,t)$, $SR^{-1}E_{\rho,\mu}(-St^{\rho})g(x,t)\in  C_x^a(\overline{\Omega})$ and $\frac{\partial^{2}}{\partial x^{2}}R^{-1}E_{\rho, \mu} (-t^{\rho} S)g(x,t)\in C_x^a({\Omega})$. Moreover,the following estimates hold:
\begin{equation}\label{RES}
||R^{-1}E_{\rho, \mu} (-t^{\rho} S)g(x,t)||_{C(\overline{\Omega})} \leq C_{3} ||g||_{C_x^a(\overline{\Omega})},
\end{equation}
\begin{equation}\label{SRES}
||SR^{-1}E_{\rho, \mu} (-t^{\rho} S)g(x,t)||_{C(\overline{\Omega})} \leq C_{4} ||g||_{C_x^a(\overline{\Omega})},
\end{equation}
\begin{equation}\label{ARES}
||\frac{\partial^{2}}{\partial x^{2}}R^{-1}E_{\rho, \mu} (-t^{\rho} S)g(x,t)||_{C[0, \pi]} \leq C_{5}t^{-\rho}||g||_{C_x^a(\overline{\Omega})}, \quad t>0.
\end{equation}
Here
\[
 R^{-1}=(\alpha^{2}I-A)^{-\frac{1}{2}}.
\]
\end{lem}
\begin{proof}
The proof follows ideas similar to the proof of Lemma \ref{E}.
\end{proof}
\begin{lem}\label{EAintL} Let $\alpha>0$ and $t^{1-\rho}g(x,t)\in C_x^a(\overline{\Omega})$. Then
	
	\begin{equation}\label{EAint}
		\bigg|\bigg|t^{1-\rho}\int\limits_0^t (t-\tau)^{\rho-1}  E_{\rho, \rho} (-S(t-\tau)^\rho )g(x,\tau)d\tau \bigg|\bigg| \leq Mt^{2\rho} \frac{\Gamma^{2}(\rho)}{\Gamma(2\rho)}||t^{1-\rho}g||_{C_x^a(\overline{\Omega})}.
	\end{equation}
\begin{equation}\label{EAint1}
    	\bigg|\bigg|t^{1-\rho}\int\limits_0^t (t-\tau)^{2\rho-1}  E^{2}_{\rho, 2\rho} (-\alpha(t-\tau)^\rho)g(x,\tau)d\tau \bigg|\bigg| \leq Mt^{3\rho} \frac{(2+\rho)\Gamma(\rho)\Gamma(2\rho)}{\rho\Gamma(3\rho)}||t^{1-\rho}g||_{C_x^a(\overline{\Omega})}.
\end{equation}
	
\begin{equation}\label{EAint2}
		\bigg|\bigg|\int\limits_0^t (t-\tau)^{\rho-1}  E^{2}_{\rho, \rho} (-\alpha(t-\tau)^\rho )g(x,\tau)d\tau \bigg|\bigg| \leq 2Mt^{3\rho-1} \frac{\Gamma^{2}(\rho)}{\rho\Gamma(2\rho)}||t^{1-\rho}g||_{C_x^a(\overline{\Omega})}.
	\end{equation}

	\end{lem}

\begin{proof}

Apply estimate (\ref{ES}) to get
    \[
    	\bigg|\bigg|t^{1-\rho}\int\limits_0^t (t-\tau)^{\rho-1}  E_{\rho, \rho} (-S(t-\tau)^\rho)g(x,\tau)d\tau \bigg|\bigg| \leq M t^{1-\rho}\int_{0}^{t}(t-\tau)^{\rho-1}\tau^{\rho-1} d\tau\cdot ||t^{1-\rho}g||_{C_x^a(\overline{\Omega})}.
    \]
For the integral one has
\begin{equation}\label{int}
	\int_{0}^{t}(t-\tau)^{\rho-1}\tau^{\rho-1} d\tau=t^{3\rho-1}\frac{\Gamma^{2}(\rho)}{\Gamma(2\rho)}.
	\end{equation}
this implies the assertion of the (\ref{EAint}). (\ref{EAint1}) and (\ref{EAint2}) are obtained in the same way as in the proof of (\ref{EAint}) combining.
\end{proof}

\begin{cor}\label{g1g2} If  function $g(x,t)$ can be represented in the form $g_1(x) g_2(t)$, then the right-hand side of estimate (\ref{EAint}) has the form:
	\[
	 Mt^{2\rho} \frac{\Gamma^{2}(\rho)}{\Gamma(2\rho)}||g_1||_{C^a[0,\pi]}||t^{1-\rho}g_2||_{C[0,T
		]}.
	\]
	
	\end{cor}

\begin{lem}\label{ep}  Let $\alpha>0$ and $\lambda_{k}\neq \alpha^{2}$,  for all $k$. Then for any $t^{1-\rho}g(x,t)\in C(\overline{\Omega}) $,  we have
    \begin{equation}\label{A}
    \bigg|\bigg|\int\limits_0^t(t-\tau)^{\rho-1} \frac{\partial^2}{\partial x^2} R^{-1}E_{\rho, \rho}(-S(t-\tau)^\rho)g(x,\tau)d\tau   \bigg|\bigg|_{C(\overline{\Omega})}\leq C ||t^{1-\rho}g||_{C_x^a(\overline{\Omega})}.
    \end{equation}
    \begin{equation}\label{A1}
    \bigg|\bigg|\int\limits_0^t(t-\tau)^{\rho-1}SR^{-1}E_{\rho, \rho}(-S(t-\tau)^\rho)g(x,\tau)d\tau   \bigg|\bigg|_{C(\overline{\Omega})}\leq C ||t^{1-\rho}g||_{C_x^a(\overline{\Omega})}.
    \end{equation}
    \begin{equation}\label{A2}
    \bigg|\bigg|\int\limits_0^t(t-\tau)^{\rho-1}R^{-1}E_{\rho, \rho}(-S(t-\tau)^\rho)g(x,\tau)d\tau   \bigg|\bigg|_{C(\overline{\Omega})}\leq C ||t^{1-\rho}g||_{C_x^a(\overline{\Omega})}.
    \end{equation}

\end{lem}

\begin{proof}Let
    \[
        S_j(x,t)= \sum\limits_{k=1}^j
        \left[\int\limits_{0}^t(t-\tau)^{\rho-1} \left(\alpha-\sqrt{\alpha^{2}-\lambda_{k}}\right)E_{\rho, \rho}\left(-\left(\alpha-\sqrt{\alpha^{2}-\lambda_{k}}\right)(t-\tau)^\rho \right)g_k(\tau) d\tau\right] \lambda_k  v_k(x).
    \]

    Choose $\varepsilon$ so that $0<\varepsilon<a-1/2$ and apply the inequality (\ref{Large}) to get
    \[
    |S_j(t)|\leq C\sum\limits_{k=1}^j  \int\limits_0^t
    (t-\tau)^{\varepsilon\rho-1}\tau^{\rho-1}\lambda_k^{\varepsilon}|\tau^{1-\rho}g_k(\tau)|
    ds.
    \]
    By Lemma \ref{Zyg} we have
    \[
   |S_j(t)|\leq C ||t^{1-\rho}g||_{C_x^a(\overline{\Omega})}
    \]
    and since
     \[
    \int\limits_0^t(t-\tau)^{\rho-1} \frac{\partial^2}{\partial x^2} R^{-1}E_{\rho, \rho}(-S^{-}(t-\tau)^\rho )g(x,\tau) d\tau =\lim_{j \to \infty}S_j(t),
    \]
    this implies the assertion of the (\ref{A}). (\ref{A1}) and (\ref{A2}) are obtained in the same way as in the proof of (\ref{A}) combining.

    \end{proof}
\section{Proof of the theorem on the forward problem}

According to the Fourier method, we will seek the solution to this problem in the form
\[
u(x,t) =\sum\limits_{k=1}^\infty T_k(t) v_k(x),
    \]
    where $T_k(t)$ are the unique solutions of the problems
    \begin{equation}\label{prob.T}
        \left\{
        \begin{aligned}
            &(\partial_t^\rho)^{2}T_{k}(t) +2\alpha\partial^{\rho}_{t}T_{k}(t)+\lambda_k T_k(t) =f_k(t),\quad 0 < t \leq T;\\
            &\lim\limits_{t\rightarrow 0}J_t^{\rho-1}(\partial^{\rho}_{t}T_k(t)) =\varphi_{0k};\\
            &\lim\limits_{t\rightarrow 0}J_t^{\rho-1} T_k(t) =\varphi_{1k},
        \end{aligned}
        \right.
    \end{equation}

Lemma \ref{du} implies
\begin{equation}\label{1}T_{k}(t)=
\begin{cases}
  & T_{1k}(t),\quad \alpha^{2}\neq \lambda_{k}; \\
 & T_{2k}(t),\quad \alpha^{2}=\lambda_{k}.
\end{cases}
\end{equation}
Here
$$
T_{1k}(t)=\frac{\sqrt{\alpha^{2}-\lambda_{k}}-\alpha}{2\sqrt{\alpha^{2}-\lambda_{k}}}t^{\rho-1}E_{\rho,\rho}((-\alpha-\sqrt{\alpha^{2}-\lambda_{k}})t^{\rho})\varphi_{1k}+\frac{\sqrt{\alpha^{2}-\lambda_{k}}+\alpha}{2\sqrt{\alpha^{2}-\lambda_{k}}}t^{\rho-1}E_{\rho,\rho}((-\alpha+\sqrt{\alpha^{2}-\lambda_{k}})t^{\rho})\varphi_{1k}
$$
$$
+\frac{1}{2\sqrt{\alpha^{2}-\lambda_{k}}}\left(t^{\rho-1}E_{\rho,\rho}((-\alpha+\sqrt{\alpha^{2}-\lambda_{k}})t^{\rho})-t^{\rho-1}E_{\rho,\rho}((-\alpha-\sqrt{\alpha^{2}-\lambda_{k}})t^{\rho})  \right)\varphi_{0k}
$$
$$
+\frac{1}{2\sqrt{\alpha^{2}-\lambda_{k}}}\int_{0}^{t}(t-\tau)^{\rho-1}E_{\rho,\rho}((-\alpha+\sqrt{\alpha^{2}-\lambda_{k}})(t-\tau)^{\rho})f_{k}(\tau)d\tau
$$
$$
-\frac{1}{2\sqrt{\alpha^{2}-\lambda_{k}}}\int_{0}^{t}(t-\tau)^{\rho-1}E_{\rho,\rho}((-\alpha-\sqrt{\alpha^{2}-\lambda_{k}})(t-\tau)^{\rho})f_{k}(\tau)d\tau.
$$
\[
T_{2k}(t)=t^{\rho-1}E_{\rho,\rho}(-\alpha t^{\rho})\varphi_{1k}+\alpha t^{2\rho-1} E_{\rho,2\rho}^{2}(-\alpha t^{\rho})\varphi_{1k}+t^{2\rho-1}E_{\rho,2\rho}^{2}(-\alpha t^{\rho})\varphi_{0k}
\]
\[
+\int_{0}^{t}(t-\tau)^{2\rho-1}E^{2}_{\rho,2\rho}(-\alpha(t-\tau)^{\rho})f_{k}(\tau)d\tau.
\]
Hence the solution to problem (\ref{prob1}) has the form
\begin{equation}\label{w}
u(x,t)=\frac{1}{2}\left(\tilde{E}_{\rho,\rho}(-S^{-}t^{\rho})+ \tilde{E}_{\rho,\rho}(-S^{+}t^{\rho})\right)t^{\rho-1}\varphi_{1}(x)
\end{equation}
\[
+\frac{\alpha}{2}\left(R^{-1}\tilde{E}_{\rho,\rho}(-S^{-}t^{\rho})- R^{-1}\tilde{E}_{\rho,\rho}(-S^{+}t^{\rho})\right)t^{\rho-1}\varphi_{1}(x)
\]
\[
+\frac{1}{2}\left(R^{-1}\tilde{E}_{\rho,\rho}(-S^{-}t^{\rho})- R^{-1}\tilde{E}_{\rho,\rho}(-S^{+}t^{\rho})\right)t^{\rho-1}\varphi_{0}(x)+t^{\rho-1}E_{\rho,\rho}(-\alpha t^{\rho})\varphi_{1k_{0}}v_{k_{0}}(x)+\alpha t^{2\rho-1} E_{\rho,2\rho}^{2}(-\alpha t^{\rho})\varphi_{1k_{0}}v_{k_{0}}(x)
\]
\[
+t^{2\rho-1}E_{\rho,2\rho}^{2}(-\alpha t^{\rho})\varphi_{0k_{0}}v_{k_{0}}(x)+\int_{0}^{t}(t-\tau)^{2\rho-1}E^{2}_{\rho,2\rho}\left(-\alpha(t-\tau)^{\rho}\right)f_{k_{0}}(\tau)v_{k_{0}}(x)d\tau
\]
\[
+\frac{1}{2}\left(\int_{0}^{t}(t-\tau)^{\rho-1}\left(R^{-1}\tilde{E}_{\rho,\rho}\left(-S^{-}(t-\tau)^{\rho}\right)-R^{-1}\tilde{E}_{\rho,\rho}\left(-S^{-}(t-\tau)^{\rho}\right)\right)f(x,\tau)d\tau\right),
\]
where we denote by $$\tilde{E}_{\rho,\mu}(-St^\rho)f(x,t)=\sum\limits_{k\neq k_0} E_{\rho,\mu} (-(\alpha\pm\sqrt{\alpha^{2}-\lambda_{k}})t^{\rho}) f_{k}(t) v_{k}(x).$$
In the case when there are several indices $k\in\mathbb{N}$ such that $\alpha^2=\lambda_{k}$, we can repeat the same argument with a slight modification in finite number of terms (note, as noted above, the multiplicity of each eigenvalue $\lambda_k$  is finite).

We show that $u(x,t)$ is a solution of the problem (\ref{prob1}) in the above cases according to Definition \ref{forward}.

Estimate  $\lVert t^{1-\rho}u(x,t)\rVert_{C(\overline{\Omega})}$ using (\ref{ES}),(\ref{RES}), Corollary \ref{cor}, (\ref{Prahma}), \eqref{EAint}, \eqref{EAint1}:
\[
\lVert t^{1-\rho}u(x,t)\lVert_{C(\overline{\Omega})} \leq (M+\alpha C_{3})||\varphi_{1}||_{C^a[0,\pi]}+C_{3}||\varphi_{0}||_{C^a[0,\pi]}+(M+\frac{3M\alpha t^{\rho}}{\rho})|\varphi_{1k_{0}}|
\]
\[
+\frac{3Mt^{\rho}}{\rho}|\varphi_{0k_{0}}|+Mt^{3\rho}\frac{(2+\rho)\Gamma(\rho)\Gamma(2\rho)}{\rho\Gamma(3\rho)}||t^{1-\rho}f_{k_{0}}(t)||_{C[0,T]}+Mt^{2\rho}\frac{\Gamma^{2}(\rho)}{\Gamma(2\rho)}||t^{1-\rho }f(x,t)||_{C_x^a(\overline{\Omega})}
\]
Next we prove that this series converges after applying operator $\frac{\partial^2}{\partial x^2}$ and the derivatives $(\partial^{\rho}_{t})^{2}$, $\partial^{\rho }_{ t}$.

Let us estimate $\frac{\partial^2}{\partial x^2}u(x,t)$. If $S_{j}(x,t)$ is a partial sum of (\ref{w}), then

\[
\frac{\partial^2}{\partial x^2}S_{j}(x,t)=\frac{1}{2}\sum\limits_{\underset{k\ne {{k}_{0}}}{\mathop{k=1}}\,}^{j}\bigg[t^{\rho-1}E_{\rho,\rho}\left(-\left(\alpha-\sqrt{\alpha^{2}-\lambda_{k}}\right)t^{\rho}\right)\varphi_{1k}+t^{\rho-1}E_{\rho,\rho}\left(-\left(\alpha+\sqrt{\alpha^{2}-\lambda_{k}}\right)t^{\rho}\right)\varphi_{1k}
\]
\[
+\frac{\alpha}{\sqrt{\alpha^{2}-\lambda_{k}}}t^{\rho-1}E_{\rho,\rho}\left(-\left(\alpha-\sqrt{\alpha^{2}-\lambda_{k}}\right)t^{\rho}\right)\varphi_{1k}-\frac{\alpha}{\sqrt{\alpha^{2}-\lambda_{k}}}t^{\rho-1}E_{\rho,\rho}\left(-\left(\alpha+\sqrt{\alpha^{2}-\lambda_{k}}\right)t^{\rho}\right)\varphi_{1k}
\]
\[
+\frac{1}{\sqrt{\alpha^{2}-\lambda_{k}}}t^{\rho-1}E_{\rho,\rho}\left(-\left(\alpha-\sqrt{\alpha^{2}-\lambda_{k}}\right)t^{\rho}\right)\varphi_{0k}-\frac{1}{\sqrt{\alpha^{2}-\lambda_{k}}}t^{\rho-1}E_{\rho,\rho}\left(-\left(\alpha+\sqrt{\alpha^{2}-\lambda_{k}}\right)t^{\rho}\right)\varphi_{0k}
\]
\[
+\frac{1}{\sqrt{\alpha^{2}-\lambda_{k}}}\int_{0}^{t}(t-\tau)^{\rho-1}E_{\rho,\rho}\left(-\left(\alpha-\sqrt{\alpha^{2}-\lambda_{k}}\right)(t-\tau)^{\rho}\right)f_{k}(\tau)d\tau
\]
\[
-\frac{1}{\sqrt{\alpha^{2}-\lambda_{k}}}\int_{0}^{t}(t-\tau)^{\rho-1}E_{\rho,\rho}\left(-\left(\alpha+\sqrt{\alpha^{2}-\lambda_{k}}\right)(t-\tau)^{\rho}\right)f_{k}(\tau)d\tau\bigg]\lambda_{k}v_{k}(x)
\]
\[
+t^{\rho-1}E_{\rho,\rho}(-\alpha t^{\rho})\varphi_{1k_{0}}\lambda_{k_{0}}v_{k_{0}}(x)+\alpha t^{2\rho-1} E_{\rho,2\rho}^{2}(-\alpha t^{\rho})\varphi_{1k_{0}}\lambda_{k_{0}}v_{k_{0}}(x)
\]
\[
+t^{2\rho-1}E_{\rho,2\rho}^{2}(-\alpha t^{\rho})\varphi_{0k_{0}}\lambda_{k_{0}}v_{k_{0}}(x)+\int_{0}^{t}(t-\tau)^{2\rho-1}E^{2}_{\rho,2\rho}\left(-\alpha(t-\tau)^{\rho}\right)f_{k_{0}}(\tau)\lambda_{k_{0}}v_{k_{0}}(x)d\tau
\]
Using estimates (\ref{AES}), (\ref{ARES}) and \eqref{Prahma}, Corollary \ref{cor}, (\ref{A})  consequently for above given expression we get
\[
\lVert\frac{\partial^2}{\partial x^2}S_{j}(x,t)\lVert_{C(\Omega)}\leq t^{\rho-1}(C_{2}+\alpha t^{-\rho}C_{5})||\varphi_{1xx}||_{C^{a}[0,\pi]}+C_{5}t^{-1}||\varphi_{0xx}||_{C^{a}[0,\pi]}
\]
\[
+Mt^{\rho-1}\alpha^{2}|\varphi_{1k_{0}}|+\frac{t^{2\rho-1}M\alpha}{\rho}\left[\alpha^{2}+\frac{(2+\rho)\alpha^{2}}{\rho}\right]|\varphi_{1k_{0}}|+\frac{t^{2\rho-1}M\alpha^{2}}{\rho}\left[1+\frac{(2+\rho)}{\rho}\right]|\varphi_{0k_{0}}|
\]
\[
+C\lVert t^{1-\rho}f(x,t)\lVert_{C^{a}_{x}(\overline{\Omega})}, \quad  t>0.
\]

Let us now estimate $\partial^{\rho}_{t}u(x,t)$. If $S_{j}(x,t)$ is a partial sum of (\ref{w}), then by lemmas \ref{Dervitavi1} - \ref{Dervitavi4}  we see that

\[
\partial^{\rho}_{t}S_{j}(x,t)=\frac{1}{2}\sum\limits_{\underset{k\ne {{k}_{0}}}{\mathop{k=1}}\,}^{j}\bigg[-t^{\rho-1}\left(\alpha-\sqrt{\alpha^{2}-\lambda_{k}}\right)E_{\rho,\rho}\left(-\left(\alpha-\sqrt{\alpha^{2}-\lambda_{k}}\right)t^{\rho}\right)\varphi_{1k}
\]
\[
-t^{\rho-1}\left(\alpha+\sqrt{\alpha^{2}-\lambda_{k}}\right)E_{\rho,\rho}\left(-\left(\alpha+\sqrt{\alpha^{2}-\lambda_{k}}\right)t^{\rho}\right)\varphi_{1k}
\]
\[
-\frac{\alpha\left(\alpha-\sqrt{\alpha^{2}-\lambda_{k}}\right)}{\sqrt{\alpha^{2}-\lambda_{k}}}t^{\rho-1}E_{\rho,\rho}\left(-\left(\alpha-\sqrt{\alpha^{2}-\lambda_{k}}\right)t^{\rho}\right)\varphi_{1k}
\]
\[
+\frac{\alpha\left(\alpha+\sqrt{\alpha^{2}-\lambda_{k}}\right)}{\sqrt{\alpha^{2}-\lambda_{k}}}t^{\rho-1}E_{\rho,\rho}\left(-\left(\alpha+\sqrt{\alpha^{2}-\lambda_{k}}\right)t^{\rho}\right)\varphi_{1k}
\]
\[
-\frac{\left(\alpha-\sqrt{\alpha^{2}-\lambda_{k}}\right)}{\sqrt{\alpha^{2}-\lambda_{k}}}t^{\rho-1}E_{\rho,\rho}\left(-\left(\alpha-\sqrt{\alpha^{2}-\lambda_{k}}\right)t^{\rho}\right)\varphi_{0k}
\]
\[
-\frac{\left(\alpha+\sqrt{\alpha^{2}-\lambda_{k}}\right)}{\sqrt{\alpha^{2}-\lambda_{k}}}t^{\rho-1}E_{\rho,\rho}\left(-\left(\alpha+\sqrt{\alpha^{2}-\lambda_{k}}\right)t^{\rho}\right)\varphi_{0k}
\]
\[
-\frac{\left(\alpha-\sqrt{\alpha^{2}-\lambda_{k}}\right)}{\sqrt{\alpha^{2}-\lambda_{k}}}\int_{0}^{t}(t-\tau)^{\rho-1}E_{\rho,\rho}\left(-\left(\alpha-\sqrt{\alpha^{2}-\lambda_{k}}\right)(t-\tau)^{\rho}\right)G_{k}(\tau)d\tau
\]
\[
+\frac{\left(\alpha+\sqrt{\alpha^{2}-\lambda_{k}}\right)}{\sqrt{\alpha^{2}-\lambda_{k}}}\int_{0}^{t}(t-\tau)^{\rho-1}E_{\rho,\rho}\left(-\left(\alpha+\sqrt{\alpha^{2}-\lambda_{k}}\right)(t-\tau)^{\rho}\right)G_{k}(\tau)d\tau\bigg]v_{k}(x)
\]
\[
-t^{\rho-1}\alpha E_{\rho,\rho}(-\alpha t^{\rho})\varphi_{1k_{0}}v_{k_{0}}(x)+\alpha t^{\rho-1} E_{\rho,\rho}^{2}(-\alpha t^{\rho})\varphi_{1k_{0}}v_{k_{0}}(x)
\]
\[
+t^{\rho-1}E_{\rho,\rho}^{2}(-\alpha t^{\rho})\varphi_{0k_{0}}v_{k_{0}}(x)+\int_{0}^{t}(t-\tau)^{\rho-1}E^{2}_{\rho,\rho}\left(-\alpha(t-\tau)^{\rho}\right)f_{k_{0}}(\tau)v_{k_{0}}(x)d\tau
\]
Applying the estimates (\ref{SES1}), (\ref{SRES}), Corollary \ref{cor}, \eqref{Prahma} and (\ref{EAint2}), (\ref{A1}) for corresponding terms of above expression we have
\[
\lVert\partial^{\rho}_{t}S_{j}(x,t)\lVert_{C(\Omega)}\leq(t^{-1}M+\alpha t^{\rho-1}C_{4})||\varphi_{1}||_{C^{a}[0,\pi]}+C_{4}t^{\rho-1}||\varphi_{0}||_{C^{a}[0,\pi]}
\]
\[
+Mt^{\rho-1}\alpha|\varphi_{1k_{0}}|+\frac{2t^{\rho-1}M\alpha}{\rho}|\varphi_{1k_{0}}|+\frac{2t^{\rho-1}M}{\rho}|\varphi_{0k_{0}}|+2Mt^{3\rho-1}\frac{\Gamma^{2}(\rho)}{\rho\Gamma(2\rho)}||t^{1-\rho}G_{k_{0}}||_{C[0,T]}
\]
\[
+C\lVert t^{1-\rho}f(x,t)\lVert_{C^{a}_{x}(\overline{\Omega})}, \quad  t>0.
\]
Using estimates (\ref{SES2}) and similar ideas as in the proof of the above estimate we have
\[
\lVert t^{1-\rho}\partial^{\rho}_{t}S_{j}(x,t) \lVert_{C(\overline\Omega)}\leq C_{1}||\varphi_{1xx}||_{C^{a}[0,\pi]}+C_{4}||\varphi_{1}||_{C^{a}[0,\pi]}+C_{4}||\varphi_{0}||_{C^{a}[0,\pi]}
\]
\[
+M\alpha|\varphi_{1k_{0}}|+\frac{2M\alpha}{\rho}|\varphi_{1k_{0}}|+\frac{2M}{\rho}|\varphi_{0k_{0}}|+2Mt^{2\rho}\frac{\Gamma^{2}(\rho)}{\rho\Gamma(2\rho)}||t^{1-\rho}f_{k_{0}}||_{C[0,T]}
\]
\[
+C\lVert t^{1-\rho}f(x,t)\lVert_{C^{a}_{x}(\overline{\Omega})}.
\]

Let us prove the uniqueness of the solution. We use a standard technique based on the completeness of the set of eigenfunctions $\{v_k(x)\}$ in $L_2(0,\pi)$.

Let $u(x,t)$ be a solution to the problem
\begin{equation}\label{Unique}
\left\{
\begin{aligned}
&(\partial_t^\rho)^{2} u(x,t)+2\alpha \partial_{t}^{\rho}u(x,t)- u_{xx}(x,t) =0,\quad (x,t)\in \Omega;\\
&u(0,t)=u(\pi, t)=0, \quad 0\leq t\leq T;\\
&\lim\limits_{t\rightarrow 0}J_t^{\rho-1}(\partial^{\rho}_{t} u(x,t))=0, \quad 0\leq x\leq \pi;\\
&\lim\limits_{t\rightarrow 0}J_t^{\rho-1}u(x,t)=0, \quad 0\leq x\leq \pi,
\end{aligned}
\right.
\end{equation}
Consider the function
\[
u_{k}(t)=\int_{0}^{\pi} u(x,t)v_{k}(x)dx
\]
By definition of the solution we may write
\[
(\partial_t^\rho)^{2} u_{k}(t)=\int_{0}^{\pi}(\partial_t^\rho)^{2}u(x,t) v_{k}(x)dx=\int_{0}^{\pi}\left(-2\alpha\partial^{\rho}_{t}u(x,t)+u_{xx}(x,t)\right)v_{k}(x)dx
\]
\[
=-2\alpha \int^{\pi}_{0}\partial^{\rho}_{t}u(x,t)v_{k}(x)dx+\int^{\pi}_{0}u(x,t)v_{kxx}(x)dx=-2\alpha\partial^{\rho}_{t}u_{k}(t)-\lambda_{k}u_{k}(t)
\]
Hence, we have the following problem for $u_k(t)$:

\begin{equation}\label{Unique2}\nonumber
\begin{cases}
  & (\partial_{t}^{\rho })^{2}u_{k}(t)+2\alpha \partial_{t}^{\rho }u_{k}(t)+\lambda_{k}u(t)=0,\quad 0<t\le T; \\
 & \underset{t\to 0}{\mathop{\lim }}\,J^{\rho-1}_{t}(\partial_{t}^{\rho }u_{k}(t))=0,\\
 & \underset{t\to 0}{\mathop{\lim }}\,J^{\rho-1}_{t}u_{k}(t)=0. \\
\end{cases}
\end{equation}
Lemma \ref{du} implies that $u_k(t)\equiv 0$ for all $k$ . Consequently, due to
the completeness of the system of eigenfunctions ${v_k(x)}$, we have $u(x,t)\equiv0$, as required.

\section{Acknowledgement}
The author are grateful to A. O. Ashyralyev for posing the
problem and they convey thanks to R. R. Ashurov for discussions of
these results.
The authors acknowledge financial support from the  Ministry of Innovative Development of the Republic of Uzbekistan, Grant No F-FA-2021-424.

\end{document}